\documentclass[12pt]{amsart}
\usepackage{amssymb,latexsym,esint}
\usepackage{enumerate}
\usepackage{float}
\usepackage{frame, amssymb,latexsym,esint, physics}
\usepackage{enumerate}
\usepackage{float}
\usepackage{import}

\usepackage{tikz}
\usetikzlibrary{arrows}

\usepackage{mathtools}
\usepackage{amsfonts}
\usepackage{amssymb}
\usepackage{amsmath}
\usepackage{graphicx}
\usepackage{graphicx,color}
\usepackage[normalem]{ulem}
\usepackage{todonotes}
\usepackage{mathrsfs}

\usepackage{import}
\usepackage{xifthen}
\usepackage{pdfpages}
\usepackage{transparent}
\usepackage{yfonts}

\numberwithin{equation}{section}
\makeatletter
\@namedef{subjclassname@2020}{%
  \textup{2020} Mathematics Subject Classification}
\makeatother

\usepackage{mdframed, amsmath,amstext,amsthm,amsxtra,mathtools,amssymb}

\usepackage[bookmarksopen,bookmarksdepth=3,colorlinks,citecolor=red,pagebackref,hypertexnames=true]{hyperref}
\usepackage[backrefs,msc-links,nobysame,initials,non-sorted-cites]{amsrefs}

\usepackage[normalem]{ulem}
\usepackage{lmodern}

\usepackage{lmodern}

\usepackage[backrefs,msc-links,nobysame,initials]{amsrefs}

\allowdisplaybreaks

\newtheorem{thm}{Theorem}

\newtheorem{lem}{Lemma}

\theoremstyle{definition}

\theoremstyle{remark}

\usepackage[noframe]{showframe}
\usepackage{framed}
\usepackage{lipsum}
 {\endMakeFramed}
\definecolor{shadecolor}{gray}{0.75}

\usepackage{amsmath,amstext,amsthm,amssymb,amsxtra}
\usepackage{mathtools}
\usepackage{enumerate}

\usepackage{amsfonts}
\usepackage{amssymb}
\usepackage{amsmath}
\usepackage{graphicx}
\usepackage{graphicx,color}
\usepackage[normalem]{ulem}
\usepackage{todonotes}

\def\norm#1.#2.{\lVert#1\rVert_{#2}}
\def\Norm#1.#2.{\bigl\lVert#1\bigr\rVert_{#2}}
\def\NOrm#1.#2.{\Bigl\lVert#1\Bigr\rVert_{#2}}
\def\NORm#1.#2.{\biggl\lVert#1\biggr\rVert_{#2}}
\def\NORM#1.#2.{\Biggl\lVert#1\Biggr\rVert_{#2}}

\begin{document}
\title[Construction of Kakeya-Type Sets]{Probabilistic Construction of Kakeya-Type Sets in $\mathbb{R}^2$ associated to separated sets of directions}

\author[Paul  Hagelstein]{Paul Hagelstein}
\address{P. H.: Department of Mathematics, Baylor University, Waco, Texas 76798}
\email{\href{mailto:paul_hagelstein@baylor.edu}{paul\_hagelstein@baylor.edu}}
\thanks{P. H. is partially supported by a grant from the Simons Foundation (\#521719 to Paul Hagelstein).}

\author[Blanca Radillo-Murguia]{Blanca Radillo-Murguia}
\address{B. R.-M.: Department of Mathematics, Baylor University, Waco, Texas 76798}
\email{\href{mailto:blanca_radillo1@baylor.edu}{blanca\_radillo1@baylor.edu}}

\author{Alexander Stokolos}
\address{A. S.: Department of Mathematical Sciences, Georgia Southern University, Statesboro, Georgia 30460}
\email{\href{mailto:astokolos@GeorgiaSouthern.edu}{astokolos@GeorgiaSouthern.edu}}


\maketitle

\begin{abstract}   We provide a condition on a set of directions $\Omega \subset \mathbb{S}^1$ ensuring that the associated directional maximal operator $M_\Omega$ is unbounded on $L^p(\mathbb{R}^2)$ for every $1 \leq p < \infty$. The techniques of proof extend ideas of Bateman and Katz involving probabilistic construction of Kakeya-type sets using sticky maps and Bernoulli percolation.   
\end{abstract}
\mbox{}

\section{Introduction}
This paper addresses problems associated to the $L^p(\mathbb{R}^2)$ boundedness of directional \mbox{maximal} operators acting on measurable functions on $\mathbb{R}^2$.  In particular we provide a condition on a set of directions so that the associated maximal operator is unbounded on $L^p(\mathbb{R}^2)$ for every $1 \leq p < \infty$.   Our research extends the classical work of Nikodym \cite{nikodym} and Busemann and Feller \cite{bf1934},  who constructed  Kakeya-type sets that may be used to provide examples indicating that the directional maximal operator associated to the set of all directions in $\mathbb{S}^{1}$ is unbounded on $L^p(\mathbb{R}^2)$ for every $1 \leq p < \infty$.  It more closely relates, however, to the more recent work of Bateman and Katz \cite{KB} and Bateman \cite{Ba} that indicates how probabilistic techniques  may be used to show that certain directional maximal operators are unbounded on $L^p(\mathbb{R}^2)$ for every $1 \leq p < \infty$.   A particularly noteworthy result in \cite{KB} in this regard due to Bateman and Katz is that if $\Omega$ is the Cantor ternary set in $[0,1]$, then the associated directional maximal operator $M_\Omega$ acting on measurable functions in $\mathbb{R}^2$ is unbounded on $L^p(\mathbb{R}^2)$ for all $1 \leq p < \infty$.   The goal of this paper is to show that the primary ideas of the paper of Bateman and Katz may, with appropriate modifications, yield similar results for  sets that are not lacunary of finite order but satisfy a certain ``separation'' condition.

The paper \cite{Ba} contains a theorem asserting that, if $\Omega$ is a subset of $\mathbb{S}^1$ that is not the union of finitely many sets of finite lacunary order, then the associated maximal operator $M_\Omega$ is not bounded on $L^p(\mathbb{R}^2)$ for any $1 \leq p < \infty$.   We have recently uncovered a subtle quantitative error in the proof of this theorem that is discussed in Section 4 of this paper.    At the present time, to the best of our knowledge, the correctness of the statement of this theorem is unknown.   That being said, extension and modification of techniques in \cite{Ba} do enable us to assert for a wide class of sets of directions $\Omega \subset \mathbb{S}^{1}$ that the associated directional maximal operators $M_\Omega$ are unbounded on $L^p(\mathbb{R}^2)$ for every $1 \leq p < \infty$.


In our paper, we will associate to a given set of directions $\Omega \subset \mathbb{S}^{1}$ a lacunary value $\lambda(\Omega)$.     The definition of the lacunary value $\lambda(\Omega)$ will be very much in the spirit of Bateman's \mbox{paper \cite{Ba}.}    
In addition to the lacunary value $\lambda(\Omega)$ associated to a given set of directions $\Omega$, we will introduce the notion that a set of directions is \emph{$\eta$-separated}.   Loosely speaking, we would say that the ternary Cantor set is $\frac{1}{3}$-separated as the distance between the intervals $[0,\frac{1}{3}]$ and $[ \frac{2}{3}, 1]$ is $\frac{1}{3}$ the length of the ambient interval $[0,1]$, with a similar relation holding for subsequent intervals in the natural construction of the ternary Cantor set.  This positive ratio is crucial in the Bateman and Katz proof that the directional maximal operator associated to a Cantor set is unbounded on $L^p(\mathbb{R}^2)$ for all $1 \leq p < \infty$.   However, as we shall see, this positive ratio does \emph{not} exist for general sets of infinite lacunary value, prohibiting the type of Bernoulli $(\frac{1}{2})$ percolation argument used by Bateman and Katz in \cite{KB} to also be used in the same manner to show that if $\Omega$ is a set of directions in $\mathbb{S}^1$ with $\lambda(\Omega) = \infty$, then $M_\Omega$ is necessarily unbounded on $L^p(\mathbb{R}^2)$ for every $1 \leq p < \infty$. The main result in our paper is that, if $\Omega \subset \mathbb{S}^{1}$ is such that, for some $\eta > 0$, $\Omega$ contains $\eta$-separated subsets $\Omega_{N}$ with $\lambda(\Omega_N) = N$ for every natural number $N$, then $M_\Omega$ is unbounded on  
$L^p(\mathbb{R}^2)$ for every $1 \leq p < \infty$.


The organization of the paper is as follows.   In the second section we will define certain terminology used in the paper, indicating what we mean by the lacunary value $\lambda(\Omega)$ of a set \mbox{$\Omega \subset \mathbb{S}^{1}$}  and the directional maximal operator $M_\Omega$ associated to $\Omega$.   We will also define the $\eta$-separation condition. In this section we will state the main theorem of the paper as well as provide an overview of the structure of the main theorem, indicating, motivated by Bateman's paper, that  there exist positive constants $c_{\eta}$ and $C_{\eta, N}$ so that $\lim_{N \rightarrow \infty}C_{\eta, N} = \infty$ and so that if $\Omega$ contains an $\eta$-separated subset of lacunary value $N$, then there exist sets $K_1$ and $K_2$ in $\mathbb{R}^2$  constructed probabilistically such that $|K_1| \geq C_{\eta, N}|K_2|$ and such that $M_\Omega \chi_{K_2} > c_{\eta}$ on $K_1$.  In this section we will recall lower estimates on the measures of all  $K_1$-type sets as provided by Bateman.     Section \ref{s3} will be devoted to the probabilistic construction of a $K_2$-type set whose measure satisfies a desired upper estimate. In Section \ref{s4} we will provide, given $N$, an example of a set  $\Omega \subset \mathbb{S}^1$ that is $N$-lacunary  but such that, letting $\mathscr{T}_{\Omega}$  be the subset associated to $\Omega$ of the binary tree and defining for each sticky map $\sigma: \mathscr{B}^{h(\mathscr{T}_\Omega)} \rightarrow \mathscr{T}_\Omega$ the associated set $K_\sigma$ as in \cite{Ba}, we have $\sup_{(x,y) \in \mathbb{R}^2 \atop {x \geq 1}} Pr((x,y) \in K_\sigma) = 1 $, where the probability is taken over all such sticky maps.   This provides a counterexample to a step in the proof of  Claim 7(B) of \cite{Ba} which asserted that for all $x \geq 1$ one has $Pr((x,y) \in K_\sigma) \lesssim \frac{1}{N}$.  In \mbox{Section \ref{s5}} we will provide an example of a set $\Omega \subset \mathbb{S}^1$ that, although having infinite lacunary value,  for no $\eta > 0$  contains an $\eta$-separated set of lacunary value $N$  for every finite value of $N$.   This example, however, does not provide a counterexample to Claim 7(B) itself.  Additionally we will make concluding remarks and make suggestions for further research in this area.
\\

We wish to express our gratitude to the referees for their comments and suggestions regarding this paper.
\section{Terminology and Statement of Main Theorem}\label{s2}

Let $\Omega$ be a nonempty subset of $\mathbb{S}^{1}$.   We may associate to $\Omega$ the directional maximal operator $M_\Omega$ acting on measurable functions on $\mathbb{R}^2$ by
$$M_\Omega f(x) := \sup_{x \in R} \frac{1}{|R|} \int_R |f|\;\textup{,}$$
where the supremum is taken over the set of all open rectangles in $\mathbb{R}^2$ \mbox{containing $x$} with an edge of longest length being oriented in one of the points (directions) of $\Omega$.  

For the remainder of the paper we will assume, without loss of generality, that $\Omega \subset \mathbb{S}^{1}$ is such that, for every $\omega \in \Omega$, the line  $\ell \subseteq \mathbb{R}^2$  passing through the origin and $\omega$ intersects the line segment
$$\{(1, u) : 0 \leq u \leq 1\}\;\textup{}$$ at a point $\omega_Q$.   For convenience, we will  identify $\Omega$ with the set $$Q_\Omega := \{u  : (1, u) = \omega_Q \textup{ for some } \omega \in \Omega\}\;$$
or more simply identify $\Omega$ with a subset of $[0,1]$.  (Equivalently, we can assume $\Omega$ lies in the first octant of the plane and we identify $\Omega$ with the tangents of the associated angles.)

We now indicate how we will denote dyadic subintervals of $[0,1]$.  Let $Q_0$ denote the interval $[0,1]$.   Let $Q_{00}, Q_{01}$ denote the two  closed a.e. disjoint dyadic subintervals of $Q_0$ whose union forms $Q_0$, where all the values in $Q_{00}$ are less than or equal to any value in $Q_{01}$.   Continuing recursively, given $  j_i \in \{0,1\} $, for $1 \leq i \leq k$, we let  $Q_{0j_1 \ldots j_k 0}, Q_{0 j_1 \ldots j_k 1}$ denote the two nonoverlapping dyadic subintervals of $Q_{0 j_1 \ldots j_k}$ whose union forms $Q_{0j_1\ldots j_k}$, where all the values in  $Q_{0 j_1\ldots j_k 0}$ are less than or equal to any value in   $Q_{0 j_1\ldots j_k 1}$.   If $u$ is the binary string  $0j_1j_2\ldots j_k$, we may abbreviate the interval $Q_{0j_1j_2\ldots j_k}$ by $Q_u$.  When convenient, we will also let $u = 0j_1j_2\ldots j_k$ denote the interval $[\sum_{i=1}^k 2^{-i}j_i, \sum_{i=1}^k 2^{-i}j_i + 2^{-k}].$ 

We define the binary tree $\mathscr{B}$ to be the graph whose vertex set consists of $0$ and all finite strings of the form $0 a_1 a_2 \ldots a_k$ where each $a_i \in \{0, 1\}$, and whose edge set is the collection of unordered pairs of  vertices of the form $(0, 0 a_1)$ or $(0 a_1 \ldots a_{k-1} , 0 a_1 \ldots a_{k-1} a_k)$.


Given $\Omega \subset \mathbb{S}^{1}$, we define $\mathscr{T}_\Omega$ to be the smallest subtree of $\mathscr{B}$  containing 0 and all of the vertices of the form $0 a_1 a_2 \ldots a_k$ such that $Q_{0 a_1 a_2 \ldots a_k} \cap Q_\Omega \neq \emptyset$.

Let $\mathscr{T}$ be a subtree of $\mathscr{B}$.   Any vertex $v \in \mathscr{T}$ of the form $v = 0 a_1 \ldots a_k$ is said to be of \emph{height} $k$, and we may write $h(v)= k$.  $0 \in \mathscr{T}$ is considered to be of height $0$.   The height of a nonempty tree is the supremum of the heights of its vertices.  If $u \textup{,} v \in \mathscr{T}$, an edge in $\mathscr{B}$ exists connecting $u$ and $v$, and $h(v) = 1 + h(u)$, $u$ is considered to be a \emph{parent} of $v$ and $v$ is considered to be a \emph{child} of $u$.     If $u_j$ is a parent of $u_{j+1}$ for $j = 0, \ldots, k - 1$, then $u_j$ is considered to be an \emph{ancestor} of $u_k$ and $u_k$ is considered to be a \emph{descendant} of $u_j$.  If the vertex $u \in \mathscr{T}$ has two children in $\mathscr{T}$, then the vertex $u$ is considered to \emph{split in $\mathscr{T}$}, and we may also refer to $u$ as a \emph{splitting vertex}.  

If $\mathscr{T}$ is a subtree of $\mathscr{B}$ and $N$ is a natural number, we define $\mathscr{T}^N$ to be the truncation of $\mathscr{T}$ to all of its vertices of height less than or equal to $N$.

A ray $R$ in $\mathscr{T}$ is a (possibly infinite) maximal ordered collection of vertices $v_1, v_2, v_3, \ldots$ in $\mathscr{T}$ such that $h(v_{j+1}) = 1 + h(v_j)$.  It is maximal in the sense that the ray does not terminate at a vertex $v \in \mathscr{T}$ if $v$ has any descendants in $\mathscr{T}$.   If $v \in \mathscr{T}$, the set of rays starting at $v$ of the form $v, v_2, v_3, \ldots$ is labeled by $\mathfrak{R}_{\mathscr{T}}(v)$.

Given a tree $\mathscr{T}$ and a ray $R$ in $\mathscr{T}$, we define the splitting number $\textup{split}(R)$ of  $R$ to be the number (possibly infinite) of vertices that split in $\mathscr{T}$ that lie on $R$.   The splitting number of a vertex $v$  with  respect to a tree $\mathscr{S}$ rooted at $v$ is defined by
$$\textup{split}_\mathscr{S} (v) := \min_{R \in \mathfrak{R}_{\mathscr{S}}(v)} \textup{split}(R)\;.$$  

 The splitting number of a vertex $v$ in a tree $\mathscr{T}$ is defined by
$$\textup{split}(v) := \sup_{\mathscr{S} \subset \mathscr{T}} \textup{split}_\mathscr{S} (v)\;,$$ where the supremum is over all subtrees $\mathscr{S}$ of $\mathscr{T}$ all of whose vertices are of height at least that of $v$.   We define 
$$\textup{split} (\mathscr{T}) := \sup_{v \in \mathscr{T}} \textup{split}(v)\;.$$

Given $\Omega \subset \mathbb{S}^{1}$, we define the \emph{lacunary value} $\lambda(\Omega)$ by
$$\lambda(\Omega) := \textup{split}(\mathscr{T}_\Omega)\;.$$

Although this terminology is motivated by that of Sj\"ogren and  Sj\"olin \cite{sjsj} and Bateman \cite{Ba}, a few words of caution are in order here.  To begin with, the lacunary value does not agree with what is typically considered the lacunary order of a set.   As an example, if $\Omega = \{1/2\}$, then $\lambda(\Omega) = 1$ since $1/2$ has a binary representation of both 100000\ldots \;and \mbox{0111\ldots .}  Similarly, multiple binary representations of numbers of the form $1 / 2^j$ lead us to have that the lacunary value of the set $\{1/2, 1/4, 1/8, \ldots\}$ is 2 although this set is typically considered to have lacunary order 1.   It is for this reason that we refer to a lacunary value of a set as opposed to a lacunary order.   We suppose we could get around this issue by associating to any point in $\Omega$ a single ray, say by choosing a ray that was minimal with respect to a type of dictionary order, but this would create a certain degree of artificiality that we wish to avoid.  At any rate, the lacunary value $\lambda(\Omega)$ that we define agrees with the splitting number $\textup{split} (\mathscr{T}_\Omega)$ defined by Bateman, so our definition would seem to be reasonable. 

Again following terminology in Bateman \cite{Ba}, we state that a tree $\mathscr{T} \subset \mathscr{B}$ is \emph{lacunary of order 0} if $\mathscr{T}$ consists of a single ray containing 0, and that $\mathscr{T}$ is lacunary of order $N$ if all of the splitting vertices of $\mathscr{T}$ lie on a lacunary tree of order $N-1$.   It is understood that when we say that a tree is lacunary of order $N$ (or, more colloquially, the tree is $N$-lacunary) that the tree is not lacunary of an order lower than $N$.

If $\mathscr{P} \subset \mathscr{B}$ is lacunary of order $N$ and of finite height $h(\mathscr{P})$, we say that $\mathscr{P}$ is \emph{pruned} provided every ray in $\mathfrak{R}_{\mathscr{P}}(0)$ contains exactly one vertex $v_j$ that splits in $\mathscr{P}$ such that $\textup{split}_\mathscr{P}(v_j) = j$ for $1 \leq j \leq N$. 

Given $\Omega \subset \mathbb{S}^1$, note the lacunary value $\lambda(\Omega)$ of $\Omega$ satisfies the equality 
$$\lambda(\Omega) = \sup\left\{N : \mathscr{T}_\Omega  \;\textup{contains a lacunary tree of order}\; N\right\}.$$

Let $0 < \eta$.   We say that a tree $\mathscr{T} \subset \mathscr{B}$ is \emph{$\eta$-separated} provided for any splitting vertex $0a_1a_2\ldots a_k$ any two descendants $u$ and $v$ that are splitting vertices and lying on separate halves of the interval $Q_{0a_1a_2\ldots a_k}$ must be such that the  Euclidean distance between the intervals $Q_u$ and $Q_v$ is greater than or equal to $\eta$ times the length of the interval $Q_{0a_1a_2\ldots a_k}$. 
\\

We are now in position to state the main theorem of the paper.

\begin{thm}\label{t1}
Let $\Omega \subset \mathbb{S}^{1}$.   Suppose there exists $\eta > 0$ so that, for every natural number $N$, the tree $\mathscr{T}_\Omega$ contains an $\eta$-separated subtree that is lacunary of order $N$.   Then the maximal operator $M_\Omega$ is unbounded on $L^p(\mathbb{R}^2)$ for every $1 \leq p < \infty$.   
\end{thm}

For any natural number $N$, a function $f: \mathscr{B}^N \rightarrow \mathscr{B}^N$ is called a \emph{sticky map} provided $h(f(u)) = h(u)$ for every $u \in \mathscr{B}^N$ and moreover such that $f(u)$ is an ancestor of $f(v)$ whenever $u$ is an ancestor of $v$.
\\

Let $\mathscr{T} \subset \mathscr{B}$ be a tree of finite height whose vertices consist of a collection of vertices $\{v_j\}$, all of height $h(\mathscr{T})$, together with all of the ancestors of these vertices.  To every sticky map $\sigma:\mathscr{B}^{h(\mathscr{T})} \rightarrow \mathscr{T}$   we may associate a set $K_{\sigma} \subset \mathbb{R}^2$ defined as follows. 
\\

 Let $d_{\sigma, 0 j_1 \ldots j_{h(\mathscr{T})}} \in [0,1]$ denote the left-hand endpoint of the interval $Q_{0 k_1 \ldots k_{h(\mathscr{T})}}$, where 
 $$\sigma(0 j_1 j_2 \ldots j_{h(\mathscr{T})}) = 0 k_1 k_2 \ldots k_{h(\mathscr{T})}\;.$$  We let $\rho_{\sigma, 0 j_1  \ldots j_{h(\mathscr{T})}}$ denote the interior of the union of all lines in $\mathbb{R}^2$ passing through the interval $\{0\} \times Q_{0 j_1\ldots j_{h(\mathscr{T})}}$ oriented in the direction  $(1, d_{\sigma, 0 j_1 \ldots j_{h(\mathscr{T})}})\;.$ We define 
$$K_\sigma = \bigcup_{j_1, \ldots, j_{h(\mathscr{T})} \atop   j_i \in \{0,1\} } \rho_{\sigma, 0 j_1  \ldots j_{h(\mathscr{T})}}\;.$$  We set

$$K_{\sigma, 1} = K_\sigma \cap \left\{(x_1, x_2) \in \mathbb{R}^2: 0 \leq x_1 \leq 1\right\}$$ 
and
$$K_{\sigma, 2, \eta} = K_\sigma \cap \left\{(x_1, x_2) \in \mathbb{R}^2: \frac{1}{2\eta}\leq x_1 \leq \frac{1}{2\eta} + 1\right\}\;.$$

Table \ref{table1} and Figure \ref{fig1}   provide an example of a sticky map $\sigma: \mathscr{B}^4 \rightarrow \mathscr{B}^4$ and the associated set $K_\sigma$.

\begin{table}[ht] \caption{Sticky map $\sigma:  \mathscr{B}^4 \rightarrow \mathscr{B}^4$}\label{table1}
 \centering    
 \begin{tabular}{c c c c}  
\hline\hline                 
$t$&$ \sigma(t)$ & $t$ & $\sigma(t)$ \\ [0.5ex] 
\hline  0000 & 1100 & 1000 & 0000  \\   0001 & 1101& 1001 & 0001  \\ 0010 & 1110 & 1010  & 0011  \\ 0011 & 1111& 1011& 0010\\ 0100 & 1011& 1100& 0100\\ 0101 & 1010 & 1101 & 0100\\  0110 & 1010 & 1110 & 0101 \\ 0111 & 1010 & 1111 & 0100\\[1ex] \hline     \end{tabular} \label{table:nonlin}  \end{table}

\begin{figure}[H]
\centering
\includegraphics[width=0.55\textwidth]{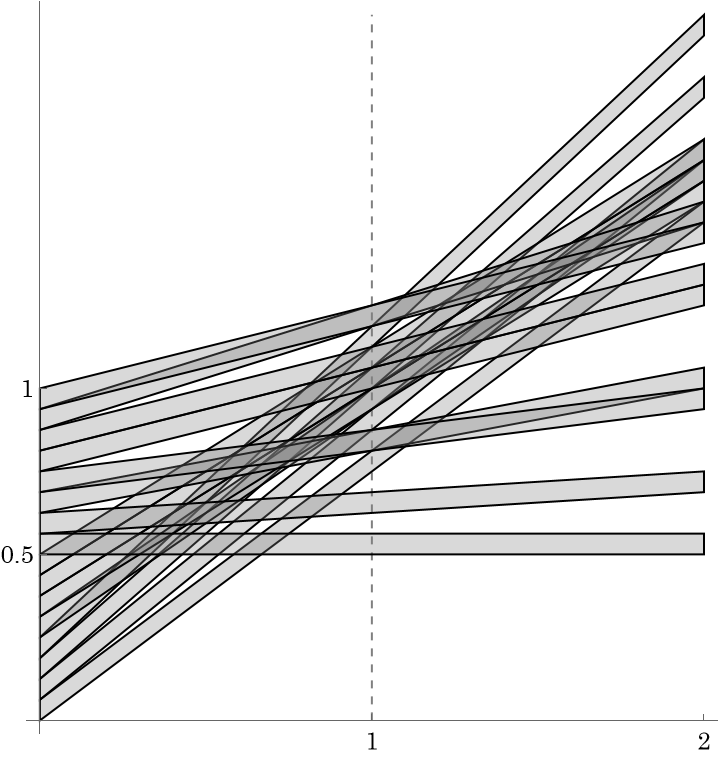}
\caption{ $K_\sigma$ set associated to Table \ref{table1}}
\label{fig1}
\end{figure}

\begin{lem}[Bateman \cite{Ba}]\label{l1}
Suppose $\mathscr{P} \subset  \mathscr{B}$ is a pruned tree that is lacunary of order $N$ and of finite height $h(\mathscr{P})$. Moreover suppose $\mathscr{P}$ contains $2^N$ vertices of height $h(\mathscr{P})$.  Then 
$$|K_{\sigma, 1}| \gtrsim \frac{\log N}{N}\;$$ holds for every sticky map $\sigma: \mathscr{B}^{h(\mathscr{P})} \rightarrow \mathscr{P}$. 
\end{lem}

\begin{lem}\label{l2}
Suppose $\mathscr{P} \subset  \mathscr{B}$ is an $\eta$-separated pruned tree that is lacunary of order $N$ and of finite height $h(\mathscr{P})$.  Moreover suppose $\mathscr{P}$ contains $2^N$ vertices of height $h(\mathscr{P})$.  Then there exists a sticky map \mbox{$\sigma: \mathscr{B}^{h(\mathscr{P})} \rightarrow \mathscr{P}$} such that 
$$|K_{\sigma, 2, \eta}| \lesssim_\eta \frac{1}{N}\;.$$
\end{lem}
 

Note Lemma \ref{l1} is essentially Claim 7A of \cite{Ba}.   Moreover, Lemma \ref{l2}, without the separation assumption, corresponds to Claim 7(B) of \cite{Ba}. To prove Theorem \ref{t1} it suffices to prove \mbox{Lemma \ref{l2}.}  To see this, suppose the hypotheses of Theorem \ref{t1} are satisfied.   By Bateman's pruning argument in Section 3 of \cite{Ba},  we have that, given $N > 0$,  there exists  a pruned tree $\mathscr{P} \subset \mathscr{T}_\Omega$ of finite height $h(\mathscr{P})$ with $2^N$ vertices of height $h(\mathscr{P})$ that is lacunary of order $N$.  As $\eta$-separation is preserved under pruning, we have $\mathscr{P}$ is $\eta$-separated.  Lemma \ref{l2} provides the existence of  a sticky map $\sigma: \mathscr{B}^{h(\mathscr{P})} \rightarrow \mathscr{P}$ such that the associated sets $K_{\sigma, 1}$ and $K_{\sigma, 2, \eta}$ satisfy
 $|K_{\sigma, 1}| \gtrsim \frac{\log N}{N}\;$ and $|K_{\sigma, 2, \eta}| \lesssim_\eta \frac{1}{N}\;.$   
Note that the average of $\chi_{K_{\sigma,2,\eta}}$ over any parallelogram $\rho_{\sigma 0j_1\ldots j_h(\mathscr{T})}\cap \{(x,y)  \in \mathbb{R}^2: 0 \leq x \leq \frac{1}{2\eta} + 1\}$ equals $\frac{1}{(\frac{1}{2\eta} + 1)}$.   The slope $d_{\sigma 0j_1\ldots j_h(\mathscr{T})}$ of this parallelogram of width  $2^{-h(\mathscr{T})}$  is within $2^{-h(\mathscr{T})}$ of a direction in $\Omega$, and accordingly is contained in a parallelogram of $\frac{1}{(\frac{1}{2\eta} + 1)}$ times its area but oriented in a direction in $\Omega$.   Hence 
$M_\Omega \chi_{K_{\sigma, 2, \eta}} \gtrsim \eta^2$ on $K_{\sigma, 1}$, and we  have the proof of Theorem \ref{t1}. 

\section{Proof of Lemma \ref{l2}}\label{s3}

\begin{proof}[Proof of Lemma \ref{l2}]  We may assume without loss of generality that $\eta = 2^{-j}$ for some natural number $j$.

Let $(x,y) \in \mathbb{R}^2$ with $\frac{1}{2\eta} < x \leq \frac{1}{2\eta} + 1$.    By linearity of  expectation, it suffices to show that $Pr((x,y) \in K_\sigma) \lesssim_\eta \frac{1}{N}$, where the probability is taken over all sticky maps $\sigma: \mathscr{B}^{h(\mathscr{P})} \rightarrow \mathscr{P}$.

 Let $ q_1, \ldots, q_k, \ldots, q_{2^N}$ denote the $2^N$ vertices in $\mathscr{P}$ of height $h(\mathscr{P})$.  For each $1 \leq k \leq 2^{N}$, we let $0q_{k1}\ldots q_{k h(\mathscr{P})}$ denote the binary string of $q_k$.   Let $b_1, \ldots, b_l$ denote all of the vertices in $\mathscr{B}$ of height $h(\mathscr{P})$ such that, if $b_k$ is the string  $0 {b_{k}}_1 \ldots {b_{k}}_{h(\mathscr{P})}$, then for some $q_{n}$ there exists a  parallelogram $\rho_k$ that contains $(x,y)$ with longest sides of slope $\sum_{j = 1}^{h(\mathscr{P})}2^{-j}{q_{n}}_j$ and with corners at $(0, \sum_{j=1}^{h(\mathscr{P})}2^{-j}{b_k}_j)$ and \mbox{$(0, \sum_{j=1}^{h(\mathscr{P})}2^{-j}{b_{k}}_j + 2^{-h(\mathscr{P})}) $} and with a right vertical side on the line $x = \frac{1}{2\eta} + 1$.   Note $0 \leq l = l(\mathscr{P},x,y) \leq 2^N$.   We assume without loss of generality that $(x,y)$ does not lie on the boundary of this parallelogram and hence there is at most one parallelogram $\rho_k$ satisfying this property.   (The set of points lying on the boundaries of all parallelograms of this form is of measure 0.)

Let $g_0$ denote the splitting vertex of $\mathscr{P}$ of lowest height.  Let $g_{00}$ denote the splitting vertex of $\mathscr{P}$ of lowest height that is or is a descendant of one child of $g_{0}$ and we let  $g_{01}$ denote the splitting vertex of $\mathscr{P}$ of lowest height that is or is a descendant of the other child.    (The reader should keep in mind that, although $g_0$ is splitting, it is possible that one or neither of its immediate children are splitting, and for this argument we need to insure that $g_{00}$ and  $g_{01}$ lie on ``opposite sides of the family tree'' rooted at $g_0$.)
More generally suppose $g_{0a_1\ldots a_j}$ has been defined for $j \leq N-2$.  We let $g_{0a_1\ldots a_j 0}$ denote the splitting vertex of $\mathscr{P}$ of lowest height that is or is a descendant of one child  of $g_{0a_1\ldots a_j}$  and we let $g_{0a_1\ldots a_j 1}$ denote the splitting vertex of $\mathscr{P}$ of lowest height that is or is a descendant of the second child.  Note the heights of both $g_{0a_1\ldots a_j 0}$ and $g_{0a_1\ldots a_j 1}$ are greater than  the height of  $g_{0a_1\ldots a_j}$  and they do not have to be equal to each other.




We now consider a splitting vertex $g_{0r_1\ldots r_k}$ of $\mathscr{P}$.  The set of real numbers $t$ such that there exits a line passing through $(x,y)$ and $(0,t)$ with slope lying in the interval $g_{0r_1\ldots r_k}$ forms an interval $I_{g_{0r_1\ldots r_k}}$ of length less than $\frac{2}{\eta}$ times the length of the interval $g_{0r_1\ldots r_k}$.  This interval is in turn contained in a union of  at most $\frac{2}{\eta} + 1$ dyadic intervals of length that of $g_{0r_1\ldots r_k}$ all of which intersecting $I_{g_{0r_1\ldots r_k}}$.  We label those intervals in [0,1] that happen to contain any of the intervals $b_1, \ldots, b_l$  as $G^m_{g_{0r_1 \ldots r_k}}$ where $m$ is an index over a set that is possibly empty but could have integer values from $1$ to as large as $\frac{2}{\eta} + 1\leq \frac{4}{\eta}$.   Since $\mathscr{P}$ satisfies an $\eta$-separability condition, if $0s_1 \ldots s_j \neq 0t_1 \ldots t_j$ no interval  $G^m_{g_{0s_1 \ldots s_j}}$ for $1 \leq j \leq N-1$  can be an interval $G^n_{g_{0t_1 \ldots t_j}}$.

  Note that if $g_{0a_1 \ldots a_j a_{j+1}}$ is a descendant of $g_{0a_1 \ldots a_j}$, then any interval $G^m_{g_{0a_1 \ldots a_j a_{j+1}}}$ must be contained in an interval  $G^n_{g_{0a_1 \ldots a_j}}$ for some $n$.    Moreover, the interval $G^k_{g_{0a_1 \ldots a_j}}$ can contain at most $\frac{4}{\eta}$ intervals  $G^m_{g_{0a_1 \ldots a_j 0}}$  and at most $\frac{4}{\eta}$ intervals  $G^m_{g_{0a_1 \ldots a_j 1}}$. 
At this stage we recall that an $n$-ary tree is that same as a binary tree except that each vertex may have up to $n$ descendants.   For example, 3-ary (or ternary) trees are considered  in \cite{KB}.

 We define an $\frac{8}{\eta}$-tree $\mathscr{G}$ described as follows.  We assume without loss of generality that $G^k_{g_0}$ exists for at least one value of $k$, as otherwise $(x,y) \notin K_\sigma$ for every sticky map $\sigma: \mathscr{B}^{h(\mathscr{P})} \rightarrow \mathscr{P}$ automatically holds.  The root of $\mathscr{G}$ is the interval $[0,1]$.  Vertices of $\mathscr{G}$ of height $h$ are the intervals $G^k_{g_{0a_1 \ldots a_h}}$, and edges are placed between $[0,1]$ and the intervals $G^k_{g_0}$ and also between any $G^m_{g_{0a_1 \ldots a_j}}$ and any  $G^n_{g_{0a_1 \ldots a_j a_{j+1}}}$.   $\mathscr{G}$ has height $N-1$.

Note that the number of vertices in $\mathscr{G}$ of height $j$ is bounded by  $ \frac{4}{\eta} 2^j$. 

The $\eta$-separation condition on $\mathscr{P}$ manifests itself at this stage of the argument in a very important way.  Namely, \emph{there do not exist two intervals  $G^m_{g_{0a_1 \ldots a_j 0}}$, $G^n_{g_{0a_1 \ldots a_j 1}}$ that lie in the same \emph{half} of an interval  $G^k_{g_{0a_1 \ldots a_j}}$.}  The reason for this is that, letting $H$ denote a half of the interval $G^k_{g_{0a_1 \ldots a_j}}$, the sets $$\left\{(u,v) \in \mathbb{R}^2 : \frac{1}{2\eta} < u, v = m u + b, b \in H, m \in g_{0a_1 \ldots a_j 0}\right\}$$ and
$$\left\{(u,v) \in \mathbb{R}^2 : \frac{1}{2\eta} < u, v = m u + b, b \in H, m \in g_{0a_1 \ldots a_j 1}\right\}$$
are disjoint, and hence both cannot simulaneously contain the point $(x,y)$.   Figures \ref{ff} and \ref{fff} illustrate the associated disjointness.

\begin{figure}

\begin{tikzpicture}[scale=7]
\draw[thick] (0,0) -- (1,0);
\foreach \x in {0,0.5,1} \draw[thick] {(\x,0.5pt) -- (\x,-0.5pt)};
\draw (0.5,-1pt) node[below] {$I$};

\draw[(-, thick, black] (0.1875,0) -- (0.25,0);
\draw[-), thick, black] (0.1875,0) -- (0.25,0);

\draw[(-, thick] (0.625,0) -- (0.75,0);
\draw[-), thick] (0.625,0) -- (0.75,0);

\draw[<-, thick] (0.25,1.5pt) -- (0.625,1.5pt);
\draw[->, thick] (0.25,1.5pt) -- (0.625,1.5pt);
\draw (0.4375,2pt) node[above] {$\eta |I|$};
\end{tikzpicture}

\caption{  $\eta$-separation  condition}
\label{ff}
\end{figure}

\;\;\;\;\mbox{   }
\begin{figure}
\centering
\includegraphics[width=0.45\textwidth]{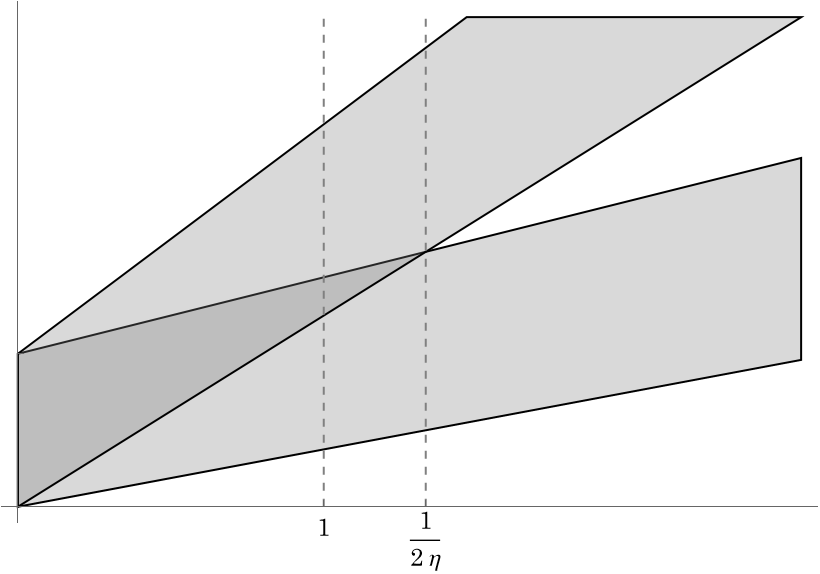}
\caption{ Disjointness past $x = \frac{1}{2\eta}$ associated to $\eta$-separation}
\label{fff}
\end{figure}

Accordingly, if  $\sigma: \mathscr{B}^{h(\mathscr{P})} \rightarrow \mathscr{P}$ is a randomly chosen sticky map, the probability that $(x,y) \in K_\sigma$ is bounded by the probability of a Bernoulli $(\frac{1}{2})$ percolation on $\mathscr{G}$, which is a subtree of an $\frac{8}{\eta}$-tree of height $N-1$ with at most  $ \frac{4}{\eta} 2^j$ vertices of height $j$.  Indeed, for $(x,y)$ to lie in the associated $K_\sigma$, a ``right choice'' for $\sigma$ must be made \emph{for a sequence of nested half-intervals associated to the vertices in a ray in $\mathscr{G}$, } a  right choice on a half-interval  associated to  a vertex in $\mathscr{G}$  being that, once made, there exists a sticky map $\psi: \mathscr{B}^{h(\mathscr{P})} \rightarrow \mathscr{P}$ agreeing with $\sigma$ on that vertex such that $(x,y) \in K_\psi$ and $K_\psi$ intersects the $y$-axis on an interval that is a subset of the interval associated to that vertex.   Note two choices for $\sigma$ are possible for every half-interval associated to a vertex in a ray in $\mathscr{G}$, but by the disjointness properties associated to the sets displayed in the previous paragraph at most one is a right one.   Note that, without the separation condition, the above disjointness property does not hold and there could be more than  one right choice.    As calculations similar to those found in \cite{KB, lyons1992} indicate (see also Exercise 5.52 (a) of \cite{lyonsperes}) this probability of this Bernoulli percolation is bounded by $\frac{C}{\eta} \frac{1}{N}$, and so the lemma holds.
\end{proof}

\section{An example}\label{s4}

Fix a natural number $N$ and $(x,y) \in \mathbb{R}^2$.  Let $\mathscr{T}$ be a pruned tree of bounded height $h(\mathscr{T})$ with lacunary order $N$.   We set $\textup{Pr}_\mathscr{T}(x,y)$ to be the probability over all sticky maps $\sigma: \mathscr{B}^{h(\mathscr{T})} \rightarrow \mathscr{T}$ that $(x,y) \in K_\sigma$.   In Claim 7(B) of \cite{Ba} it is asserted that $\textup{Pr}_{\mathscr{T}}(x,y) \lesssim \frac{1}{N}$ provided $(x,y) \in [1,2] \times [0,3]$, although this is not necessarily the case.   An example is provided by the following theorem.

\begin{thm}\label{t2}
Given a natural number $N$, there exists a pruned tree $\mathscr{P}$ of lacunary order $N$ and a point $(1,y)$ with $1 \leq y \leq 3/2$ such that $Pr_\mathscr{P}(1,y) = 1$.
\end{thm}

\begin{proof}

We define the \emph{interval maps} $\rho_1$ and $\rho_2$ on the set of closed intervals in $\mathbb{R}$ of positive finite measure by the following:

$$\rho_1([a,b]) = \left[\frac{a + b}{2} - \frac{b-a}{8}, \frac{a + b}{2}\right]  \textup{;}$$
$$\rho_2([a,b]) = \left[\frac{a + b}{2}, \frac{a + b}{2} + \frac{b - a}{8}\right] \;.$$

 
Let $\mathcal{I}_0 = [0,1]$.  We set $\mathcal{I}_{01}= \rho_1\mathcal{I}_0$ and $\mathcal{I}_{02} = \rho_2\mathcal{I}_0$.    Similarly, for any sequence $a_1, a_2, \ldots, a_k$ of 1's and 2's, we let
$$\mathcal{I}_{0a_1a_2\ldots a_k} = \rho_{a_k}\rho_{a_{k-1}}\ldots \rho_{a_1}\mathcal{I}_0\;.$$
We let $m_{0a_1 a_2 \ldots a_k}  \in [0,1]$ denote the left hand endpoint of $ \mathcal{I}_{0a_1a_2\ldots a_k}$  .   Given $N$, let $\Omega_N$ be   the set of all points of the form $m_{0a_1 a_2 \ldots a_N} + 2^{-100N}$.   Let $\mathscr{P}$ be the pruned tree of lacunary order $N$ and of height $3N$ consisting of all vertices in $\mathscr{B}$ corresponding to the intervals $\mathcal{I}_{0a_1a_2\ldots a_k}$ for $0 \leq k \leq N$.  Note $\mathscr{P} =  \mathscr{T}_{\Omega_N}^{3N},$ the tree $\mathscr{T}_{\Omega_N}$ truncated at height $3N$.

Let now $y = 1 + \frac{1}{8} + (\frac{1}{8})^2 + \cdots + (\frac{1}{8})^{N} - (\frac{1}{8})^{N+1}$.   We will show that for \emph{every} sticky map $\sigma: \mathscr{B}^{3N} \rightarrow {\mathscr{P}}$ we have $(1,y) \in K_\sigma$, proving the desired result.   In particular, we will see that the probability of $(1,y)$ lying in $K_\sigma$ does not correspond to the survivor probability of  a  Bernoulli $(\frac{1}{2})$ percolation of a binary tree of height $N$, as rather for this particular value of $y$ each sticky map $\sigma$ provides a ``tournament'' for which exactly one slope in $\Omega_N$ is associated to a parallelogram  in $K_{\sigma}$ that contains $(1,y)$.   This is because, for every natural number $j$ and sequence $a_1a_2\ldots a_{j}$, the translates of both  $\mathcal{I}_{0a_1a_2\ldots a_j1}$ and $\mathcal{I}_{0a_1a_2\ldots a_j2}$ by $\frac{1}{8} + \dots + (\frac{1}{8})^{j+1}$ lie in the \emph{same half} of a dyadic interval in $[0,1]$ whose length is the same as that of the interval $\mathcal{I}_{0a_1a_2\ldots a_{j}}$.  
\\

We now discuss the above remarks in detail.   It will be helpful to associate to any set $\mathcal{S}$ in $[0,1]$ its reflection across $\frac{1}{2}$ that we denote by $rf (\mathcal{S})$; in particular 
$$rf (\mathcal{S}) = \{1 - x : x \in \mathcal{S}\}\;.$$  

 We will associate to each $\mathcal{I}_{0a_1a_2\ldots a_k}$ its reflection  $\mathcal{J}_{0a_1a_2\ldots a_k}$ given by
$$\mathcal{J}_{0a_1a_2\ldots a_k} = rf(\mathcal{I}_{0a_1a_2\ldots a_k})\;.$$

If $h \in \mathbb{R}$ and $\mathcal{S} \subset \mathbb{R}$, we define the translate $\tau_h \mathcal{S}$ by $\chi_{\tau_h \mathcal{S}}(x) = \chi_\mathcal{S}(x-h)$.
\\

If a line with slope lying in $\mathcal{S} \subset [0,1]$ intersects the point $(1,1)$, then that line intersects the $y$-axis at a point $(0,t)$ where $t \in rf(\mathcal{S})$. If a line with slope lying in $\mathcal{S} \subset [0,1]$ intersects the point $(1,1+h)$, then that line intersects the $y$-axis at a point $(0,t)$ where $t$ lies in the set $h + rf(\mathcal{S}) = \tau_h rf(\mathcal{S}) $.    Note that for each interval $\mathcal{I}_{0a_1\ldots a_N}$ there exists a parallelogram with slope $m_{0a_1\ldots a_N}$ whose left hand side is the set $\{(0, t) : t\in \mathcal{J}_{0a_1 \ldots a_N}\}$ and such that $(1,1)$ is on the upper edge. Note that consequently there exists a parallelogram with slope $m_{0a_1\ldots a_N}$  whose left hand side is the vertically oriented interval $\{(0, t) : t\in \tau_{-1 + y + (\frac{1}{8})^{N+1}}\mathcal{J}_{0a_1 \ldots a_N}\}$ that contains the point $(1,y)$ at a distance $(\frac{1}{8})^{N+1}$ vertically below its top edge.  In Bateman's terminology in \cite{Ba}, it is these $2^N$ intervals that are associated to the possible set $\textup{Poss}(1,y)$ and the associated tree $\langle(1,y)\rangle$.   These intervals being  important enough to us to give them a name, we let $$\mathcal{K}_{0a_1 \ldots a_N} = \tau_{-1 + y + (\frac{1}{8})^{N+1}}\mathcal{J}_{0a_1 \ldots a_N}\;.$$
\\
We also define for any string $a_1 \ldots a_j$ of $1$'s and $2$'s for $1 \leq j \leq N$ the set $$\mathcal{K}_{0a_1a_2\ldots a_j} = \tau_{
 \sum_{k=1}^j 2^{-3k}}\mathcal{J}_{0a_1a_2\ldots a_j}\;.$$

Note that for $1 \leq j \leq N-1$  \emph{both} $\mathcal{K}_{0a_1a_2\ldots a_j 1}$ and $\mathcal{K}_{0a_1a_2\ldots a_j 2}$ lie on the right half of $\mathcal{K}_{0a_1a_2\ldots a_j}$\;.
\\

Our choice of $\Omega_N$ and $y$ gives the possible set $\textup{Poss}(1,y)$ a particular structure that we now wish to exploit.
\\


  Let us now fix a sticky map $\sigma: \mathscr{B}^{3N} \rightarrow {\mathscr{P}}$ .   We need to show that $(1,y)$ must lie in $K_\sigma$.  Readers familiar with Bateman's terminology might at this point observe that since $[0,1]$ and all intervals of the form $\mathcal{I}_{0a_1a_2\ldots a_j}$ for $1 \leq j \leq N-1$ correspond to splitting vertices of $\mathscr{P}$,  we have that $[0,1]$ and intervals of the form $\mathcal{K}_{0a_1a_2\ldots a_j}$ for $1 \leq j \leq N-1$ are associated to choosing vertices of $\langle(1,y)\rangle$.  
  
   Associated to $\sigma$ will be the string of 1's and 2's that are defined recursively as follows.
   
   Let $k_1$ be such that $\sigma[\frac{1}{2}, 1] \supset \mathcal{I}_{0 k_1}\;.$ Note that $\sigma(\mathcal{K}_{0k_1}) = \mathcal{I}_{0k_1}$ because $  {\mathscr{P}}$ has no splitting vertices between $\mathcal{I}_0$ and $\mathcal{I}_{0 k_1}$.
   
   Assuming $k_1, \ldots, k_j$ are determined, let $k_{j+1}$ be the value such that
    $$\sigma(\textup{right half }(\mathcal{K}_{0k_1k_2\ldots k_j})) \;\supset\; \mathcal{I}_{0k_1k_2\ldots k_{j+1}}\;.$$
    Note that $\sigma(\mathcal{K}_{0k_1k_2\ldots k_{j+1}}) = \mathcal{I}_ {0k_1k_2\ldots k_{j+1}}$ since there are no splitting vertices in $ {\mathscr{P}}$ between $
 \mathcal{I}_ {0k_1k_2\ldots k_{j}}$ and $\mathcal{I}_ {0k_1k_2\ldots k_{j+1}}$.
 
We have that $(1,y)$ lies in the parallelogram in $K_\sigma$ with slope $m_{0k_1\ldots k_N}$ whose left hand side is the vertically oriented interval $\{(0,t) : t \in \mathcal{K}_{0k_1\ldots k_N}\}$.
\end{proof}

\section{Concluding remarks}\label{s5}
Even though Theorem \ref{t2} provides a counterexample to a step in the proof of Theorem 1 of \cite{Ba}, it does not provide a counterexample to Theorem 1 of \cite{Ba} itself as the union of all of the trees $\mathscr{T}_{\Omega_j}$ in the proof of Theorem \ref{t2} contains, for every natural number $N$, a subtree that is $\frac{1}{8}$-separated and lacunary of order $N$.
\\

Theorem 1 of \cite{Ba} was used as a key step in the proofs of Theorems 2,3, and 4 of \cite{h2013} and Theorem 4 of \cite{hs2}.   Claim 7(B) of \cite{Ba} was used in the proofs of Theorems 2.1 and 2.2 of \cite{gauvannyj2024} and Theorem 1 of \cite{hs2022}\;.  Researchers in the area should consider these results at the moment to be at best provisional.
\\


It is important to recognize that a set $\Omega \subset [0,1]$ may have infinite lacunary value yet be such that for no $\eta > 0$ does $\mathscr{T}_\Omega$ contain an $N$-lacunary $\eta$-separated subtree for every $N$.    Such a set may be defined as follows.

Let $j \in \mathbb{N}$.   We define the interval maps $\rho_{j,1} , \rho_{j,2}$ respectively on the set of closed intervals in $\mathbb{R}$ of finite length by

$$\rho_{j,1}([a,b]) = \left[\frac{a + b}{2} - 2^{-j}(b-a), \frac{a + b}{2}\right]\;,$$
$$\rho_{j,2}([a,b]) = \left[\frac{a + b}{2} , \frac{a + b}{2} + 2^{-j}(b-a)\right]\;.$$

Let $\Omega \subset [0,1]$ be the set of all $x$ contained  in infinitely many intervals of the form

$$\rho_{1,i_1}\rho_{2,i_2}\cdots \rho_{k,i_k}[0,1]\;.$$

One can check that $\lambda(\Omega) = \infty$.   However, there is no $\eta > 0$ for which  $\mathscr{T}_\Omega$ contains an $N$-lacunary $\eta$-separated subtree for every $N$. 
\\

This set $\Omega$ constructed here poses a model problem for the theory of geometric maximal operators.  Is the geometric maximal operator $M_\Omega$ bounded on $L^p(\mathbb{R}^2)$ for every $1 < p  \leq \infty$\;?

\end{document}